\documentclass[a4paper,leqno]{amsart}

\usepackage{latexsym}
\usepackage[english]{babel}
\usepackage{fancyhdr}
\usepackage[mathscr]{eucal}
\usepackage{amsmath}
\usepackage{mathrsfs}
\usepackage{amsthm}
\usepackage{amsfonts}
\usepackage{amssymb}
\usepackage{amscd}
\usepackage{bbm}
\usepackage{graphicx}
\usepackage{graphics}
\usepackage{latexsym}
\usepackage{color}

\newcommand{\ud}{\mathrm{d}}

\newcommand{\re}{\mathfrak{Re}}

\newcommand{\ii}{\mathrm{i}}
\newcommand{\cH}{\mathcal{H}}

\newcommand{\C}{\mathbb C}

\newcommand{\R}{\mathbb R}

\DeclareMathOperator{\sinc}{sinc}

\theoremstyle{plain}
\newtheorem{theorem}{Theorem}[section]
\newtheorem{lemma}[theorem]{Lemma}

\theoremstyle{definition}
\newtheorem{definition}[theorem]{Definition}

\newtheorem*{remark*}{Remark}

\numberwithin{equation}{section}

\begin{document}

\title[On real resonances for 3D point interactions]{On real resonances for three-dimensional Schr\"odinger operators with point interactions}
\author[A.~Michelangeli]{Alessandro Michelangeli}
\address[A.~Michelangeli]{Institute for Applied Mathematics, University of Bonn \\ Endenicher Allee 60 \\ 
D-53115 Bonn (Germany).}
\email{michelangeli@iam.uni-bonn.de}
\author[R.~Scandone]{Raffaele Scandone}
\address[R.~Scandone]{Gran Sasso Science Institute -- GSSI \\ viale Francesco Crispi 7 \\ 67100 L'Aquila (Italy).}
\email{raffaele.scandone@gssi.it}

\begin{abstract}
We prove the absence of positive real resonances for Schr\"odinger operators with finitely many point interactions in $\R^3$ and we discuss such a property from the perspective of dispersive and scattering features of the associated Schr\"odinger propagator.
\end{abstract}

\date{\today}

\subjclass[2000]{}

\maketitle

\section{Introduction and background: dispersive properties of the point-singular perturbed Schr\"{o}dinger equation}\label{sec:intro-background}

A typical obstacle to the dispersive and scattering properties of the time evolution group associated with the Schr\"{o}dinger equation
\begin{equation}\label{eq:SEq}
 \ii\partial_t u\;=\;-\Delta u + V u
\end{equation}
in the unknown $u\equiv u(t,x)$, where $t\in\mathbb{R}$, $x\in\mathbb{R}^d$ ($d\in\mathbb{N}$), and $V:\mathbb{R}^d\to\mathbb{R}$ is a given measurable potential, is the existence of non-trivial solutions to 
\begin{equation}\label{eq:EVeqn}
 -\Delta u + V u\;=\;\mu u
\end{equation}
for some $\mu\in\mathbb{R}$.

In those cases, relevant in a variety of contexts, where $V$ is sufficiently localised (`with short range') and/or is a suitably small perturbation of the Laplacian, the existence of non-trivial $L^2(\mathbb{R}^d)$- solutions to \eqref{eq:EVeqn} are interpreted as bound states of the associated Schr\"{o}dinger operator, and if $\mu> 0$ one refers to it as an eigenvalue embedded in the continuum. Solutions to \eqref{eq:EVeqn} in weaker $L^2$-weighted spaces are generally known instead as \emph{resonances} (a notion that we shall explicitly define in due time for the purposes of the present analysis), and they too affect the dispersive and scattering behaviour of the propagator defined by \eqref{eq:SEq}.

When $V\in L^2_{\mathrm{loc}}(\mathbb{R}^d)$, and in \eqref{eq:EVeqn} $u\in H^2_{\mathrm{loc}}(\mathbb{R}^d)$, it was first proved by Kato \cite{Kato_growth_properties-1959} that positive eigenvalues are absent, and by Agmon  \cite{Agmon-annSNS-1975} and by Alshom and Smith \cite{Alsholm-Schmidt-1971} that positive resonances are absent too. For rougher (non-$L^2_{\mathrm{loc}}$) potentials, positive eigenvalues were excluded by Ionescu and Jerison \cite{Ionescu-Jerison-2003} and by Koch and Tataru \cite{Koch-Tataru-2005} by means of suitable Carleman-type estimates which imply, owing to a unique continuation principle \cite{Jerison-Kenig-1985,Koch-Tataru-2001}, that the corresponding eigenfunctions must be compactly supported and hence vanish. Absence of positive resonances whose associated resonant state $u$ (solution to \eqref{eq:EVeqn}) satisfies appropriate radiation conditions at infinity, was proved by Georgiev and Visciglia \cite{Georgiev-Visciglia_resonances-2007} for $L^{d/2}_{\mathrm{loc}}$-potentials decaying as $|x|^{-(1+\varepsilon)}$ or faster.

A closely related and equally challenging context, which this work is part of, is the counterpart problem of existence or non-existence of spectral obstructions (eigenvalues or resonances) when the potential $V$ in \eqref{eq:SEq}-\eqref{eq:EVeqn} is formally replaced by a finite number of delta-like bumps localised at certain given points in space. This corresponds to a well-established rigorous construction of point-like perturbations of the free Laplacian usually referred to as `Schr\"{o}dinger operator with point interaction' (we refer to the monograph \cite{albeverio-solvable} as the standard references in this field).

We recall that among the various equivalent, yet conceptually alternative ways of defining on $L^2(\mathbb{R}^d)$ the formal operator
\begin{equation}\label{eq:formalHamilt}
^\textrm{``}-\Delta+\sum_{j=1}^N\,\nu_j\,\delta(x-y_j)^\textrm{''}
\end{equation}
obtained by adding to the free Laplacian $N$ singular perturbations centred at the points $y_1,\dots,y_N\in\mathbb{R}^d$ and of magnitude, respectively, $\nu_1,\dots,\nu_N\in\mathbb{R}$, one is to obtain \eqref{eq:formalHamilt} as the limit of Schr\"{o}dinger operators with actual potentials $V^{(j)}_\varepsilon(x-y_j)$ each of which, as $\varepsilon\downarrow 0$, spikes up to a delta-like profile, the support shrinking to the point $\{y_j\}$, and another way is to define \eqref{eq:formalHamilt} as a self-adjoint extension of the restriction of $-\Delta$ to smooth functions supported away from the $y_j$'s. Either approaches reproduce the free Laplacian unless when $d=1,2,3$, in which case one obtains a non-trivial perturbation of $-\Delta$.

In this work we shall focus on $d=3$ spatial dimensions and therefore fix a collection
\begin{equation*}
Y\;:=\;\{y_1, \dots, y_N\}\;\subset\;\mathbb{R}^3
\end{equation*}
of $N$ distinct points where the perturbation is supported at. It turns out \cite[Section II.1.1]{albeverio-solvable}  that the operator
\begin{equation*}
-\Delta\upharpoonright C_0^\infty(\mathbb{R}^3 \!\setminus\! Y)
\end{equation*}
is densely defined, real symmetric, and non-negative on $L^2(\mathbb{R}^d)$, and admits a $N^2$-real-parameter family of self-adjoint extensions, each of which acts as the free negative Laplacian on functions that are supported away from the interaction centres. The most relevant extensions constitute the $N$-parameter sub-family
\[
\{-\Delta_{\alpha,Y}\,|\,
\alpha\equiv(\alpha_1,\dots,\alpha_N)\in(-\infty,\infty]^N\}
\]
of so-called `local' extensions, that is, the rigorous version of \eqref{eq:formalHamilt}: for them the functions $u$ in the domain of self-adjointness are only qualified by certain local boundary conditions at each singularity centre, with no pairing between distinct centres, which take the explicit form
\begin{equation*}
\lim_{r_j\downarrow 0} 
\Big(\frac{\partial(r_j u)}{\partial r_j}- 4\pi \alpha_j r_j u\Big) =0\,,\quad r_j:=|x-y_j|\,, \quad j\in\{1, \dots, N\}\,. 
\end{equation*}

Physically, each $\alpha_j$ is proportional to the inverse scattering length of the interaction supported at $y_j$. In particular, if for some $j\in\{1,\ldots ,N\}$ one has $\alpha_j=\infty$, then no actual interaction is present at the point $y_j$, and in practice things are as if one discards it. When $\alpha=\infty$, one recovers the the Friedrichs extension of $-\Delta\upharpoonright C_0^\infty(\mathbb{R}^3 \!\setminus\! Y)$, namely the self-adjoint negative Laplacian with domain $H^2(\R^3)$. We may henceforth assume, without loss of generality, that $\alpha$ runs over $\R^N$.

The perturbations $-\Delta_{\alpha,Y}$ of $-\Delta$ have a long history of investigation and in Section \ref{sec:setup-results} we shall list a number of properties that rigorously qualify them  and are relevant for our subsequent analysis.
In the mathematical literature they were introduced and characterised for the case $N=1$ by Berezin and Faddeev \cite{Berezin-Faddeev-1961}, Albeverio, H\o{}egh-Krohn, and Streit \cite{Albeverio-HK-Streit-1977}, Nelson \cite{Nelson-1977}, Albeverio, Fenstad, and H\o{}egh-Krohn \cite{Albeverio-Fenstad-HoeghKrohn-1979_singPert_NonstAnal}, and Albeverio and H\o{}egh-Krohn \cite{AHK-1981-JOPTH}.
For generic $N\geqslant 1$ centres, $-\Delta_{\alpha,Y}$ was rigorously studied first by Albeverio, Fenstad, and H\o{}egh-Krohn \cite{Albeverio-Fenstad-HoeghKrohn-1979_singPert_NonstAnal}, and subsequently characterised by Zorbas \cite{Zorbas-1980}, 
Grossmann, H\o{}egh-Krohn, and Mebkhout 
\cite{Grossmann-HK-Mebkhout-1980,Grossmann-HK-Mebkhout-1980_CMPperiodic}, D{\setbox0=\hbox{a}{\ooalign{\hidewidth\lower1.5ex\hbox{`}\hidewidth\crcr\unhbox0}}}browski and Grosse \cite{Dabrowski-Grosse-1985}, and more recently by Arlinski\u{\i} and Tsekanovski\u{\i} \cite{arlinski-tsekanoviski-2005}, and by Goloshchapova, Malamud, and Zastavnyi \cite{Goloshchapova-Zastavnyi-Malamud-2010,Goloshchapova-Zastavnyi-Malamud-2011}.

The analysis of the dispersive and scattering properties of the Schr\"{o}dinger propagator $e^{\ii t\Delta_{\alpha,Y}}$, $t\in\mathbb{R}$, has been an active subject as well. A class of $L^p\to L^q$ dispersive estimates were established by D'Ancona, Pierfelice, and Teta \cite{DAncona-Pierfelice-Teta-2006} (in weighted form), and by Iandoli and Scandone \cite{Iandoli-Scandone-2017} (removing the weights used in \cite{DAncona-Pierfelice-Teta-2006} in the largest regime possible of the $(p,q)$-indices). The $L^p$-boundedness of the wave operators for the pair $(-\Delta_{\alpha,Y},-\Delta)$ in the regime $p\in(1,3)$ (from which dispersive and Strichartz estimates can be derived by intertwining $-\Delta$ and $-\Delta_{\alpha,Y}$), as well as the $L^p$-unboundedness of the wave operators when $p=1$ or $p\in[3,+\infty]$, was proved by Dell'Antonio, Michelangeli, Scandone, and Yajima \cite{DMSY-2017} (with counterpart results by Duch\^{e}ne, Marzuola, and Weinstein \cite{Duchene-Marzuola-Weinstein-2010} in $d=1$ and Cornean, Michelangeli, and Yajima \cite{CMY-2018-2Dwaveop} in $d=2$ dimensions).

In analogy with the ordinary Schr\"{o}dinger equation \eqref{eq:SEq}, also the dispersive features of the singular point-perturbed Schr\"{o}dinger equation
\begin{equation}\label{eq:SEq-sing}
 \ii\partial_t u\;=\;-\Delta_{\alpha,Y} \,u
\end{equation}
strictly depend on the possible presence of eigenvalues or resonances for $-\Delta_{\alpha,Y}$, and indeed in the above-mentioned works \cite{DAncona-Pierfelice-Teta-2006,Iandoli-Scandone-2017,DMSY-2017} special assumptions on the choice of $\alpha$ and $Y$ are often made so as to ensure that no spectral obstruction occurs.

In fact (see Theorem \ref{thm:general_properties} below for the complete summary and references), the spectrum $\sigma(-\Delta_{\alpha,Y})$ only consists of an absolutely continuous component $[0,+\infty)$ which is also the whole essential spectrum, plus possibly a number of non-positive eigenvalues. Thus, as usual, for the purposes of the dispersive analysis, one considers $P_{\mathrm{ac}}\,e^{\ii t\Delta_{\alpha,Y}}$, namely the action of the singular Schr\"{o}dinger propagator on the sole absolutely continuous subspace of $L^2(\mathbb{R}^3)$, and additionally one has to decide whether possible resonances are present.

When $N=1$ the picture is completely controlled: $-\Delta_{\alpha,Y}$ has only one negative eigenvalue if $\alpha<0$, and has only a resonance, at zero, if $\alpha=0$; correspondingly the integral kernel of the propagator $e^{\ii t\Delta_{\alpha,Y}}$ is explicitly known, as found by Scarlatti and Teta \cite{Scarlatti-Teta-1990} and Albeverio, Brze\'{z}niak, and D{\setbox0=\hbox{a}{\ooalign{\hidewidth\lower1.5ex\hbox{`}\hidewidth\crcr\unhbox0}}}browski \cite{Albeverio_Brzesniak-Dabrowski-1995}, from which $L^p\to L^q$ dispersive estimates are derived directly, as found in \cite{DAncona-Pierfelice-Teta-2006}. In certain regimes of $p,q$ slower decay estimates do emerge in the resonant case $\alpha=0$, as opposed to the non-resonant one.

For  generic $N$ perturbation centres, it is again well understood (see Theorem \ref{thm:general_properties} below) that at most $N$ non-positive eigenvalue can add up to the absolutely continuous spectrum $[0,+\infty)$ of $-\Delta_{\alpha,Y}$. In particular, as discussed by one of us in \cite[Sect.~3]{Scandone_ResExp_2019}, a zero-energy eigenvalue may occur (see also \cite[page 485]{albeverio-solvable}).

The study of resonances for generic $N$ has been attracting a considerable amount of attention. As explained in Section \ref{sec:setup-results}, it is known since the already mentioned work \cite{Grossmann-HK-Mebkhout-1980_CMPperiodic} by Grossmann, H\o{}egh-Krohn, and Mebkhout (see also \cite[Sect.~II.1.1]{albeverio-solvable}), that resonances and eigenvalues $z^2$ of $-\Delta_{\alpha,Y}$ are detected, on an equal footing, by the singularity of an auxiliary $N\times N$ square matrix $\Gamma_{\alpha,Y}(z)$ depending on $z\in\mathbb{C}$. \emph{Real negative resonances} (thus $z=\ii\lambda$ with $\lambda>0$) are excluded by the arguments of \cite{Grossmann-HK-Mebkhout-1980_CMPperiodic}. A \emph{zero resonance} may occur, and one of us \cite{Scandone_ResExp_2019} qualified this possibility in terms of a convenient low-energy resolvent expansion which is at the basis of our definition \ref{def:resonances} below. \emph{Complex resonances} ($\mathfrak{Im}z<0)$ have been investigated by Albeverio and Karabash \cite{Albeverio-Karabash-2017,Albeverio-Karabash-2018-multilevelReson,Albeverio-Karabash-2018-GenAsymptReson} and Lipovsk\'{y} and Lotoreichik \cite{Lipovsky-Lotoreichik-2017}, using techniques on the localisation of zeroes of exponential polynomials, and turn out to lie mostly within certain logarithmic strips in the complex $z$-plane. \emph{Real positive resonances} (thus, $z\in\mathbb{R}\setminus\{0\}$) have been 
recently excluded by Galtbayar and Yajima \cite{Galtbayar-Yajima-2019}, and implicitly also by Goloshchapova, Malamud, and Zastavnyi \cite{Goloshchapova-Zastavnyi-Malamud-2010,Goloshchapova-Zastavnyi-Malamud-2011}.

In this work we supplement this picture by demonstrating the \emph{absence of positive resonances} for $-\Delta_{\alpha,Y}$ with an argument that has 
the two-fold virtue of being particularly compact as compared to the general setting of \cite{Goloshchapova-Zastavnyi-Malamud-2010,Goloshchapova-Zastavnyi-Malamud-2011}, and exploiting the explicit structure of the matrix $\Gamma_{\alpha,Y}(z)$, unlike the abstract reasoning of \cite{Galtbayar-Yajima-2019} (further comments in this respect are cast at the end of Section \ref{sec:proofs}): as such, the approach that we present here has its own autonomous interest.

Moreover, we have already mentioned that the absence of positive resonances for an ordinary Schr\"{o}dinger operator $-\Delta+V$ is typically proved with Carleman's estimate, whereas for the singular version $-\Delta_{\alpha,Y}$ it appears to be very hard to use those classical techniques -- and indeed our proof relies on a direct analysis based on the explicit formula for the resolvent: this makes any proof of absence of resonances surely valuable.

In Section \ref{sec:setup-results} we present the rigorous context within which our main result is formulated. In particular, we survey the definition and the basic properties of the singular point-perturbed Schr\"{o}dinger operator $-\Delta_{\alpha,Y}$ and we formulate the precise definition of resonance.

The proof of our main theorem is then discussed in Section \ref{sec:proofs}, together with a few additional comments for comparison with the previous literature.

We conclude our presentation in Section \ref{sec:remarks_open_problems} with some final remarks that connect our main theorem with recent dispersive and scattering results for the Schr\"{o}dinger evolution of the singular point-perturbed Laplacian, and highlight interesting open questions.

\medskip

\textbf{Notation.} For vectors in $x,y\in\mathbb{R}^d$ the Euclidean norm and scalar products shall be denoted, respectively, by $|x|$ and $x\cdot y$, whereas for the action of an operator (or a matrix, in particular) $A$ on the vector $v$ we shall simply write $Av$. The expression $\delta_{j,k}$ denotes the Kronecker delta. By $\mathbbm{1}$ and $\mathbb{O}$ we shall denote, respectively, the identity and the zero operator, irrespectively of which vector space they act on, which will be clear from the context. By $\overline{z}$ and $Z^*$ we shall denote, respectively, the complex conjugate of a scalar $z\in\mathbb{Z}$ and the transpose conjugate of a square matrix $Z$ with complex entries. We shall use the shortcut $\langle x\rangle:=\sqrt{1+x^2}$ for $x\in\mathbb{R}$. By $\mathcal{B}(X,Y)$ we shall denote the space of bounded linear operators from the Banach space $X$ to the Banach space $Y$. For a vector $\psi$ in a Hilbert space $\cH$ the $\cH\to\cH$ rank-one orthogonal projection onto the span of $\psi$ shall be indicated with $|\psi\rangle\langle\psi|$. The rest of the notation is standard or will be declared in due time.

\section{Set-up and main result}\label{sec:setup-results}

Let us start by collecting an amount of well known facts concerning the three-dimensional singular point-perturbed Schr\"{o}dinger operator that we informally referred to in the course of the previous Section.

Let us fix $N\in\mathbb{N}$, a collection $Y=\{y_1,\dots,y_N\}$ of distinct points in $\mathbb{R}^3$, and a multi-index parameter $\alpha\equiv(\alpha_1,\dots,\alpha_N)\in\mathbb{R}^N$.

For $z\in \mathbb{C}$ and $x,y,y'\in \mathbb{R}^3$, let us set 
\begin{equation}\label{eq:def_of_the_Gs}
\mathcal{G}_z^y(x)\;:=\;\frac{e^{\ii z|x-y|}}{\,4\pi |x-y|\,},\qquad
\mathcal{G}_z^{yy'}\;:=\;\begin{cases}
\displaystyle\frac{e^{\ii z|y-y'|}}{\,4\pi |y-y'|\,} & \textrm{if }\;y'\neq y \\ 
\qquad 0 & \textrm{if }\;y'= y\,, \\
\end{cases}
\end{equation}
and 
\begin{equation}\label{ga-def}
\Gamma_{\alpha,Y}(z)\;:=\;\Big(\Big(\alpha_j-\frac{\ii z}{\,4\pi\,}\Big)\delta_{j,k}-\mathcal{G}_z^{y_jy_k}\Big)_{\!j,k=1,\dots,N}\,.
\end{equation}
Clearly, the map $z\mapsto \Gamma_{\alpha,Y}(z)$ has values in the space of $N\times N$ symmetric, complex valued matrices, and is entire. Therefore, $z\mapsto \Gamma_{\alpha,Y}(z)^{-1}$ is meromorphic on $\mathbb{C}$ and hence the subset $\mathcal{E}_{\alpha,Y}\subset\mathbb{C}$ of poles of $\Gamma_{\alpha,Y}(z)^{-1}$ is discrete. Let us further define
\begin{equation}\label{eq:EpmE0}
 \begin{split}
  \mathcal{E}_{\alpha,Y}^\pm\;&:=\;\mathcal{E}_{\alpha,Y}\cap\mathbb{C}^\pm \\
   \mathcal{E}_{\alpha,Y}^0\;&:=\;\mathcal{E}_{\alpha,Y}\cap\mathbb{R}\,, 
 \end{split}
\end{equation}
where $\mathbb{C}^+$ (resp., $\mathbb{C}^-$) denotes as usual the open complex upper (resp., lower) half-plane. 


\begin{definition}\label{def:DaY}
 Let $z\in\mathbb{C}^+\setminus\mathcal{E}_{\alpha,Y}^+$. The operator $-\Delta_{\alpha,Y}$ is defined on the domain
 \begin{equation}\label{eq:domain_of_HaY}
  \mathcal{D}(-\Delta_{\alpha,Y})\;:=\;\left\{
  u\in\,L^2(\R^3)\left|\!
  \begin{array}{c}
   u\,=\,F_z + \displaystyle\sum_{j,k=1}^N (\Gamma_{\alpha,Y}(z)^{-1})_{jk} \, F_z(y_k) {\mathcal{G}}_{z}^{y_j}  \\
   \textrm{for some }F_z \in H^2(\mathbb{R}^3)
  \end{array}
  \!\!\!\right.\right\}
 \end{equation}
 by the action
 \begin{equation}\label{eq:action_of_H}
 (-\Delta_{\alpha,Y}-z^2{\mathbbm{1}})\,u\;=\;(-\Delta -z^2{\mathbbm{1}})\,F_z\,.
\end{equation}
\end{definition}

It is straightforward to check that at fixed $z$ the decomposition \eqref{eq:domain_of_HaY} of a generic element in $\mathcal{D}(-\Delta_{\alpha,Y})$ is unique, and that the space $ \mathcal{D}(-\Delta_{\alpha,Y})$, as well as the action of $-\Delta_{\alpha,Y}$ on a generic function of its domain, are actually independent of the choice of $z$. Moreover, on $H^2$-functions $F$ vanishing at all points of $Y$ one has $-\Delta_{\alpha,Y}F=-\Delta F$.

\begin{theorem}[Basic properties of the point-perturbed Schr\"{o}dinger operator]\label{thm:general_properties}~
\begin{itemize}
 \item[(i)] The operator $-\Delta_{\alpha,Y}$ is self-adjoint on $L^2(\mathbb{R}^3)$ and extends the operator $-\Delta\upharpoonright C_0^\infty(\mathbb{R}^3 \!\setminus\! Y)$. The Friedrichs extension of the latter, namely $-\Delta$ with domain $H^2(\mathbb{R}^3)$, corresponds to the formal choice $\alpha=\infty$ in Definition \ref{def:DaY}.
 \item[(ii)] If $u\in\mathcal{D}(-\Delta_{\alpha,Y})$ and $u|_{\mathcal{U}}=0$ for some open subset $\mathcal{U}\subset\mathbb{R}^3$, then $(-\Delta_{\alpha,Y}u)|_{\mathcal{U}}=0$.
 \item[(iii)] The set $\mathcal{E}_{\alpha,Y}^+$ of poles of $\Gamma_{\alpha,Y}(z)^{-1}$ in the open complex half-plane consists of at most $N$ points that are all located along the positive imaginary semi-axis, and for $z\in\mathbb{C}^+\setminus\mathcal{E}^+$ one has the resolvent identity
\begin{equation}\label{eq:resolvent_identity}
(-\Delta_{\alpha,Y} -z^2{\mathbbm{1}})^{-1} -(-\Delta-z^2{\mathbbm{1}})^{-1} \;=\; 
\sum_{j,k=1}^N  (\Gamma_{\alpha,Y}(z)^{-1})_{jk} \,|\mathcal{G}_{z}^{y_j}\rangle\langle
\overline{\mathcal{G}_{z}^{y_k}}|\,.
\end{equation}
\item[(iv)] The essential spectrum  of $-\Delta_{\alpha,Y}$ is purely absolutely continuous and coincides with the non-negative half-line, the singular continuous spectrum is absent, and there are no positive eigenvalues:
\[
\begin{split}
 \sigma_{\mathrm{ess}}(-\Delta_{\alpha,Y}) \;&=\;\sigma_{\mathrm{ac}}(-\Delta_{\alpha,Y}) \;=\;[0,+\infty) \\
 \sigma_{\mathrm{sc}}(-\Delta_{\alpha,Y})\;&=\;\emptyset \\
 \sigma_{\mathrm{p}}(-\Delta_{\alpha,Y})\;&\subset\; (-\infty,0]\,.
\end{split}
\]
\item[(v)] There is a one-to-one correspondence between the poles $z=\ii\lambda\in\mathcal{E}_{\alpha,Y}^+$ of $\Gamma_{\alpha,Y}(z)^{-1}$ and the 
negative eigenvalues $-\lambda^2$ of $-\Delta_{\alpha,Y}$, counting the 
multiplicity. The eigenfunctions associated with 
the eigenvalue $-\lambda^2<0$ have the form
\[
 u\;=\;\sum_{j=1}^N c_j\,\mathcal{G}_{\ii\lambda}^{y_j},
\]
where $(c_1,\dots,c_N)\in\ker\Gamma_{\alpha,Y}(\ii\lambda)$. In the special case $N=1$ $(Y=\{y\})$,
\[
 \begin{split}
  \sigma_p(-\Delta_{\alpha,Y})\;=\;
  \begin{cases}
   \emptyset & \textrm{if }\alpha\geqslant 0 \\
   \{-(4\pi\alpha)^2\} & \textrm{if }\alpha <0\,,
  \end{cases}
 \end{split}
\]
and the unique negative eigenvalue, when it exists, is non-degenerate and with eigenfunction $\mathcal{G}_{-4\pi\ii\alpha}^{y}$.
\end{itemize} 
\end{theorem}

Theorem \ref{thm:general_properties} is a collection of classical results from \cite{Zorbas-1980,Grossmann-HK-Mebkhout-1980,Grossmann-HK-Mebkhout-1980_CMPperiodic}, which are discussed in detail, e.g., in \cite[Sect.~II.1.1]{albeverio-solvable}.


In addition to Theorem \ref{thm:general_properties}, the spectral behavior of $-\Delta_{\alpha,Y}$ on the real line, and in particular the nature of the spectral point $z^2=0$, was discussed by one of us in \cite{Scandone_ResExp_2019}, and we shall now review those results.

Let us first remark, as emerges from Theorem \ref{thm:general_properties}, that the eigenvalue zero is absent when $N=1$, but may occur when $N\geqslant 2$: examples of configurations of the $y_j$'s that produce a null eigenvalue are shown in \cite[Sect.~3]{Scandone_ResExp_2019}.

In \cite{Scandone_ResExp_2019} a limiting absorption principle for $-\Delta_{\alpha}$ was established, in the spirit of the classical Agmon-Kuroda theory for the free Laplacian \cite{Agmon-annSNS-1975,Kuroda1978_intro_scatt_theory}, and a low-energy resolvent expansion was produced, analogously to the case of regular Schr\"odinger operators with scalar potential \cite{Agmon-annSNS-1975,Jensen-Kato-1979}.


\begin{theorem}[\cite{Scandone_ResExp_2019}]\label{th:res_exp}
Let $\sigma>0$ and let $\mathbf{B}_{\sigma}$ be the Banach space
\begin{equation*}
\mathbf{B}_{\sigma}\;:=\;\mathcal{B}(L^2(\R^3,\langle x\rangle^{2+\sigma}\ud x),L^2(\R^3,\langle x\rangle^{-2-\sigma}\ud x))\,.
\end{equation*}
\begin{itemize}
 \item[(i)] For every $z\in\C^+\setminus\mathcal{E}_{\alpha,Y}^+$ one has $(-\Delta_{\alpha,Y}-z^2\mathbbm{1})^{-1}\in\mathbf{B}_{\sigma}$, and the map $\C^+\setminus\mathcal{E}_{\alpha,Y}^+\ni z\mapsto (-\Delta_{\alpha,Y} -z^2\mathbbm{1})^{-1}\in \mathbf{B}_{\sigma}$ can be continuously extended to $\R\setminus\mathcal{E}_{\alpha,Y}^0$.
 \item[(ii)] In a real neighborhood of $z=0$, one has the expansion
\begin{equation}\label{main_exp}
(-\Delta_{\alpha,Y}-z^2\mathbbm{1})^{-1}\;=\;z^{-2}R_{-2}+z^{-1}R_{-1}+R_0(z)\,,
\end{equation}
for some $R_{-2},R_{-1}\in \mathbf{B}_{\sigma}$ and some continuous $\mathbf{B}_{\sigma}$-valued map $z\mapsto R_0(z)$.
Moreover, $R_{-2}\neq \mathbbm{O}$ if and only if zero is an eigenvalue for $-\Delta_{\alpha,Y}$. 
\end{itemize}
\end{theorem}

In view of Theorem \ref{th:res_exp}(ii) and of the heuristic idea of a resonance as the existence of a non-$L^2$ solution $u$ to $-\Delta_{\alpha,Y}u=z^2u$, it is natural to say that $-\Delta_{\alpha,Y}$ has a \emph{zero} resonance when $(-\Delta_{\alpha,Y}-z^2\mathbbm{1})^{-1}=O(z^{-1})$ as $z\to 0$, that is, with respect to the low-energy asymptotics \eqref{main_exp}, when $R_{-2}=\mathbbm{O}$ and $R_{-1}\neq \mathbbm{O}$.

There is an equivalent way to formulate such an occurrence (the proof of which is also deferred to Section \ref{sec:proofs}).

\begin{lemma}\label{lem:AiffB}
 The following facts are equivalent:
 \begin{itemize}
  \item[(i)] in the asymptotics \eqref{main_exp}, $R_{-2}=\mathbbm{O}$ and $R_{-1}\neq \mathbbm{O}$;
  \item[(ii)] the matrix $\Gamma_{\alpha,Y}(0)$ is singular, but zero is not an eigenvalue of $-\Delta_{\alpha,Y}$.
 \end{itemize}
\end{lemma}

The resolvent identity \eqref{eq:resolvent_identity} and Lemma \ref{lem:AiffB} above finally motivate the following precise notion of resonance.

\begin{definition}\label{def:resonances}
 Let $z\in\mathbb{C}$. The operator $-\Delta_{\alpha,Y}$ has a resonance at $z^2$ if the matrix $\Gamma_{\alpha,Y}(z)$ is singular, but $z^2$ is not an eigenvalue of $-\Delta_{\alpha,Y}$.
\end{definition}

As announced in previous Section's introduction, in this work we focus on the \emph{exclusion of positive resonances} -- informally speaking, to draw a parallel to \eqref{eq:EVeqn}, we shall conclude that the equation
\begin{equation}\label{eq:EVeqn-perturbed}
 -\Delta_{\alpha,Y} u \;=\;\mu u\qquad\qquad (\mu>0)
\end{equation}
admits no non-trivial solutions, be they in $L^2(\mathbb{R}^3)$ (non-existence of embedded eigenvalues, as seen already in Theorem \ref{thm:general_properties}(iv)) or outside of $L^2(\mathbb{R}^3)$ (non-existence of embedded resonances). More precisely, in this work we prove that for any $z\in\mathbb{R}\setminus\{ 0\}$, the spectral point $\mu=z^2$ is not a resonance.

In view of Definition \ref{def:resonances} and Theorem \ref{thm:general_properties}(iv), the absence of positive resonances is tantamount as the non-singularity of $\Gamma_{\alpha,Y}(z)$ for any $z\in\mathbb{R}\setminus\{ 0\}$. This is precisely the form of our main result.

\begin{theorem}\label{th:main}
For every $\alpha\in\R^N$, every collection $Y=\{y_1,\ldots y_N\}$ of $N$ distinct points in $\mathbb{R}^3$, and every $z\in\mathbb{R}\setminus\{ 0\}$, the matrix $\Gamma_{\alpha,Y}(z)$ is non-singular. Equivalently, the self-adjoint operator $-\Delta_{\alpha,Y}$ has no  real positive resonances.
\end{theorem}

Theorem \ref{th:main}, combined with Theorem \ref{th:res_exp}, essentially completes the picture of the spectral theory for three-dimensional Schr\"odinger operators with finitely many point interactions.

For completeness of presentation, at the end of the proof in the following Section, we shall comment on the comparison with the previous literature on the positive resonances of $-\Delta_{\alpha,Y}$, and in the subsequent Section \ref{sec:remarks_open_problems} we shall connect our main theorem with recent dispersive and scattering results for the propagator $^{\ii t \Delta_{\alpha,Y}}$, together with some relevant open problems.

\section{Proof of the main Theorem and additional remarks}\label{sec:proofs}

This Section is mainly devoted to the proof of Theorem \ref{th:main}. In terms of the notation \eqref{eq:EpmE0}, one has to prove that the set $\mathcal{E}_{\alpha,Y}^0\setminus\{0\}$ of non-zero poles of $\Gamma_{\alpha,Y}(z)^{-1}$ on the real line is empty.

In practice it suffices to only consider $z>0$, for $\Gamma_{\alpha,Y}(-z)=\Gamma_{\alpha,Y}(z)^*$ for any $z\in\mathbb{R}$, and hence $\mathcal{E}_{\alpha,Y}^0$  is symmetric with respect to $z=0$.

In fact, $\mathcal{E}_{\alpha,Y}^0$ is also finite. Indeed, $\mathcal{E}_{\alpha,Y}^0\subset\mathcal{E}_{\alpha,Y}$ is a discrete set, and for $z\in\mathbb{R}$ one has $\Gamma_{\alpha,Y}(z)=-\frac{\ii z}{4\pi}\mathbbm{1}+\Lambda_{\alpha,Y}(z)$ where the matrix norm of $\Lambda_{\alpha,Y}(z)$ is uniformly bounded in $z$, therefore $\Gamma_{\alpha,Y}(z)$ is invertible for large enough $z$.

Let us first present the proof of Lemma \ref{lem:AiffB}, which was at the basis of the definition of resonance for $-\Delta_{\alpha,Y}$.

\begin{proof}[Proof of Lemma \ref{lem:AiffB}]
We recall \cite[Section 2]{Jensen-Kato-1979} that $(-\Delta-z^2\mathbbm{1})^{-1}\in\mathbf{B}_{\sigma}$ for every $z\in\C^+$, and the map $\C^+\ni z\mapsto (-\Delta -z^2\mathbbm{1})^{-1}\in \mathbf{B}_{\sigma}$ can be continuously extended to the real line. Moreover, we observe that the map $\R\ni z\to |\mathcal{G}^{y_1}_z\rangle\langle\overline{\mathcal{G}^{y_2}_z}|\in\mathbf{B}_{\sigma}$ is continuous for any $y_1,y_2\in\R^3$. Owing to these facts, one compares the limits $z\to 0$ in the resolvent identity \eqref{eq:resolvent_identity}, in the resolvent expansion \eqref{main_exp}, and in in the low-energy expansion
$$\Gamma_{\alpha,Y}(z)^{-1}\;=\;z^{-2}A_{-2}+z^{-1}A_{-1}+O(1)$$
established in \cite[Proposition 5]{Scandone_ResExp_2019}, and concludes
\begin{equation}\tag{*}\label{eq:iffs}
R_{-2}\;\neq\; \mathbbm{O}\;\Leftrightarrow\; A_{-2}\;\neq\; \mathbbm{O}\,,\qquad R_{-1}\;\neq\; \mathbbm{O}\;\Leftrightarrow\; A_{-1}\;\neq\; \mathbbm{O}\,.
\end{equation}
We can now prove the desired equivalence.

(i)$\,\Rightarrow\,$(ii). Since $R_{-2}=\mathbbm{O}$, Theorem \ref{th:res_exp}(ii) guarantees that $z=0$ is not an eigenvalue for $-\Delta_{\alpha,Y}$. Moreover, since $R_{-1}\neq \mathbbm{O}$, by \eqref{eq:iffs} also $A_{-1}\neq \mathbbm{O}$, which implies in particular that $\Gamma_{\alpha,Y}(0)$ is singular.

(ii)$\,\Rightarrow\,$(i). Since $z=0$ is not an eigenvalue for $-\Delta_{\alpha,Y}$, Theorem \ref{th:res_exp}(ii) guarantees that $R_{-2}= \mathbbm{O}$, whence also $A_{-2}=\mathbbm{O}$ owing to \eqref{eq:iffs}. Since $\Gamma_{\alpha,Y}(0)$ is singular, necessarily $A_{-1}\neq \mathbbm{O}$, whence also $R_{-1}\neq \mathbbm{O}$ again owing to \eqref{eq:iffs}.
\end{proof}

Our argument for Theorem \ref{th:main} is based upon the following useful result in linear algebra. We denote by $ Sym_{N}(\R)$ the space of $N\times N$ symmetric real matrices.

\begin{lemma}\label{linear}
Let $A,B \in Sym_{N}(\R)$,  and assume furthermore that $B$ is positive definite. Then $A-\ii B$ is non-singular.
\end{lemma}

\begin{proof}
Suppose for contradiction that $A-\ii B$ is singular. Then there exist $v,w\in\R^n$, at least one of which is non-zero, such that
\begin{equation*}
(A-\ii B)(v+\ii w)\;=\;0\,.
\end{equation*}
We can exclude for sure that $v=0$, for in this case $Bw+\ii Aw=0$, whence in particular $Bw=0$ and therefore $w=0$, against the assumption that $v+\ii w\neq 0$. 
Applying $B^{-1}$ to the identity above, separating real and imaginary parts, and setting $C:=B^{-1}A$, one gets
\[
 Cv=-w\,,\qquad Cw=v\,,
\]
which implies $C^2v=-v$. As $v\neq 0$, the conclusion is that $-1$ is an eigenvalue for $C^2$. However, the matrix $C=B^{-1}A$ can be diagonalised over $\R$, since it is similar to the symmetric matrix $B^{-1/2}AB^{-1/2}$, and therefore $C^2$ is similar to a positive semi-definite matrix. Hence $-1$ cannot be in the spectrum of $C^2$.
\end{proof}

We shall also make use of the following property.

\begin{lemma}\label{lem_ab_exist}
Let $N,d\in\mathbb{N}$ and let  $y_1,\dots,y_N$ be distinct vectors in $\mathbb{R}^d$. Then there exists a unit vector $a\in\mathbb{R}^d$ such that the numbers $a\cdot v_j$ are all distinct.
\end{lemma}

\begin{proof}
Let $\sigma_d$ be the area measure on the unit sphere $\mathbb{S}^{d-1}:=\{x\in\R^d\;|\;|x|=1\}$. For every pair $(j,k)$, with $j,k\in\{1,\ldots,N\}$,  $j\neq k$, let us consider the set 
$$P_{jk}\;:=\;\{x\in\R^d\;|\;\,x\cdot(y_j-y_k)=0\}\,.$$
Since $y_j\neq y_k$, $P_{jk}$ is an hyperplane in $\R^d$, whence $\sigma_d(\mathbb{S}^{d-1}\cap P_{jk})=0$. It follows that the set
$$Q\;:=\;\mathbb{S}^{d-1}\;\setminus\bigcup_{\substack{ j,k\in\{1,\ldots ,N\} \\ j\neq k}}(\,\mathbb{S}^{d-1}\cap P_{jk})$$
satisfies $\sigma_{d}(Q)=\sigma_{d}(\mathbb{S}^{d-1})$, and in particular $Q$ is non-empty. If we choose $a\in Q$, then by construction the numbers $a\cdot y_j$ are all distinct.
\end{proof}

We are ready to prove the main Theorem.

\begin{proof}[Proof of Theorem \ref{th:main}]
Let $z>0$: owing to the scaling property
\[
 \Gamma_{\alpha,Y}(\lambda z)\;=\;\lambda\,\Gamma_{\lambda^{-1}\alpha,\lambda Y}(z)
\]
valid for every $\lambda>0$, it is enough to prove that $\Gamma_{\alpha,Y}:=\Gamma_{\alpha,Y}(1)$ is non-singular for any choice of $\alpha$ and $Y$.

Using \eqref{eq:def_of_the_Gs}-\eqref{ga-def}, we re-write
\[
 \Gamma_{\alpha,Y}\;=\;A_{\alpha,Y}-\ii B_Y
\]
with $A_{\alpha,Y},B_Y\in Sym_N(\R)$ given explicitly by
\begin{equation*}
\begin{split}
(A_{\alpha,Y})_{jk}\;&:=\;\alpha_j\delta_{j,k}-\re\,\mathcal{G}_1^{y_jy_k}\\
(B_Y)_{jk}\;&:=\;\frac{1}{4\pi}\,\sinc(|y_j-y_k|)\,,
\end{split}
\end{equation*}
where the real function $\sinc(x)$ is defined by
$$\sinc(x)\;:=\;\begin{cases}
\displaystyle\frac{\sin(x)}{x} & \textrm{if }x\neq 0\\
\;\;1 & \textrm{if } x=0\,.
\end{cases}$$ 
Owing to Lemma \ref{linear}, the thesis follows when one proves that $B_Y$ is positive definite.

Based on an immediate integration in polar coordinates, it is convenient to express
$$\sinc(|x|)\;=\;\frac{1}{4\pi}\int_{\mathbb{S}^2}e^{ixp}\,\ud\sigma_2(p),\qquad x\in\R^3\,,$$
where $\sigma_2(p)$ is the area measure on the unit sphere $\mathbb{S}^2\subset\R^3$.

For generic $v\equiv(v_1,\dots,v_N)$
let $B_Y[v]:=v\cdot B_Y v$ be the quadratic form associated to $B_Y$.
Since
\begin{equation*}
\begin{split}
B_Y[v]\;&=\;\sum_{j,k=1}^N\frac{1}{4\pi}v_jv_k\, \sinc(|y_j-y_k|)\\
&=\;\frac{1}{\,16\pi^2}\sum_{j,k=1}^N v_jv_k\int_{\mathbb{S}^2}e^{i(y_j-y_k)\cdot p}\,\ud\sigma_2(p)\\
&=\;\frac{1}{\,16\pi^2}\int_{\mathbb{S}^2}\bigg|\sum_{j=1}^N v_j \,e^{iy_j\cdot p}\bigg|^2 \ud\sigma_2(p)\;\geqslant\; 0\,,
\end{split}
\end{equation*}
then the matrix $B_Y$ is positive semi-definite.

To demonstrate that $B_Y$  is actually positive definite, we specialise to the present context the clever argument by Castel, Filbir, and Szwarc \cite{Castel-Filbir-Szwarc-2004-jat2005} (in  \cite{Castel-Filbir-Szwarc-2004-jat2005} the general question of linear independence of exponential maps over subsets of $\R^d$ is addressed).

Assume that for some $v\in\mathbb{R}^3\setminus\{0\}$ one has $B_Y v=0$. From the above computation of $B_Y[v]$, one deduces
\begin{equation*}\tag{*}\label{secret}
\sum_{j=1}^N v_j\, e^{iy_j\cdot p}\;=\;0 \qquad \forall p\in \mathbb{S}^2.
\end{equation*}

We show, by induction on $N$, that the latter identity implies $v=0$. The case $N=1$ is obvious. Let $N\geqslant 2$, and in \eqref{secret} let us consider all possible $p\in\mathbb{S}^2$ of the form
\[
 p\;=\; a\sin t+b\cos t\,,\qquad t\in[0,2\pi)
\]
for two fixed vectors $a,b\in\mathbb{R}^3$ such that $|a|=|b|=1$, $a\perp b$, and the scalars $\alpha_j:= y_j\cdot a$ with $j\in\{1,\dots,N\}$ are all distinct. Lemma \ref{lem_ab_exist} ensures that this choice of $a$ and $b$ is possible. It is non-restrictive to assume $\alpha_1<\cdots<\alpha_N$, and let us also set $\beta_j:= y_j\cdot b$, $j\in\{1,\dots,N\}$. Then \eqref{secret} reads
\[
 \sum_{j=1}^N v_j\, e^{i(\alpha_j\sin t+ \beta_j \cos t)}\;=\;0\qquad \forall t\in[0,2\pi)\,.
\]
In fact, since l.h.s.~above depends analytically on $t$, such an identity holds true for every $t\in\C$. Specialising it for $t=-\ii\tau$, $\tau\in\mathbb{R}$, it takes the form
\[
 \sum_{j=1}^N v_j\, e^{\alpha_j\sinh\tau+ \ii\beta_j\cosh\tau}\;=\;0\qquad \forall \tau\in\mathbb{R}\,,
\]
and also, upon dividing by $\exp(\alpha_N\sinh\tau+ \ii\beta_N\cosh\tau)\neq 0$,
\[
 \sum_{j=1}^N v_j\, e^{(\alpha_j-\alpha_N)\sinh\tau+ \ii(\beta_j-\beta_N)\cosh\tau}\;=\;0\qquad \forall \tau\in\mathbb{R}\,.
\]
Since $\alpha_j<\alpha_N$ for $j<N$, taking in the latter expression $\tau$ arbitrarily large and positive  implies necessarily $v_N=0$. By the inductive assumption that \eqref{secret} implies $v=0$ when it is considered with $N-1$ instead of $N$, one concludes that also $v_1=\cdots=v_{N-1}=0$.
\end{proof}

In the remaining part of this Section we comment on how the absence of positive resonances for $-\Delta_{\alpha,Y}$ could be also read out from the already mentioned recent works \cite{Goloshchapova-Zastavnyi-Malamud-2010,Goloshchapova-Zastavnyi-Malamud-2011,Galtbayar-Yajima-2019}.

In \cite{Galtbayar-Yajima-2019} the reasoning is based upon the computation of the residue of $\Gamma_{\alpha,Y}(z)^{-1}$ at a generic pole $z\in\mathcal{E}_{\alpha,Y}^0$. From the resolvent identity \eqref{eq:resolvent_identity} \emph{and} the information that $(-\Delta_{\alpha,Y}-z^2\mathbbm{1})^{-1}$ is bounded for $z\in\mathbb{R}\setminus\{0\}$ it is shown that one obtains instead a zero value, thus contradicting the fact that $z$ is a pole. This approach requires the additional knowledge (already available, as seen in Theorem \ref{thm:general_properties}(iv)) that $-\Delta_{\alpha,Y}$ has no positive eigenvalues, and by-passes the explicit structure of the matrix $\Gamma_{\alpha,Y}(z)$.

In \cite{Goloshchapova-Zastavnyi-Malamud-2010,Goloshchapova-Zastavnyi-Malamud-2011} strictly speaking no reference to (positive) resonances of $-\Delta_{\alpha,Y}$ is made. The operator $-\Delta_{\alpha,Y}$ and its main properties are recovered by means of the alternative framework of boundary triplets and Weyl function. Then the positive definiteness of what we denoted here by $B_Y$ is indirectly alluded to by considering another matrix, with similar structure, and proving for the latter the positive definiteness by means of general properties of positive definite functions like our $\mathrm{sinc}(x)$.

\section{Connection with the dispersive properties of $e^{\ii t\Delta_{\alpha,Y}}$ and open problems}\label{sec:remarks_open_problems}

In this short, concluding Section we return to the general subject of the dispersive properties of the singular point-perturbed Schr\"{o}dinger equation \eqref{eq:SEq-sing}, in order to emphasize the connection of our Theorem \ref{th:main} with recent dispersive and scattering results for the propagator $e^{\ii t\Delta_{\alpha,Y}}$.

We already mentioned in Section \ref{sec:intro-background} that $L^1\to L^\infty$ (and hence by interpolation general $L^p\to L^q$) dispersive estimates for $e^{\ii t\Delta_{\alpha,Y}}$ were first proved in \cite{DAncona-Pierfelice-Teta-2006}, with a suitable weight that accounts for the singularity $|x-y_j|^{-1}$ of $e^{\ii t\Delta_{\alpha,Y}}f$, whereas in the subsequent work \cite{Iandoli-Scandone-2017} reproduced such estimates without weight in the regime $q\in[2,3)$. In \cite{DAncona-Pierfelice-Teta-2006} the authors assumed that $\Gamma_{\alpha,Y}(z)$ be non-singular for every $z\in\mathbb{R}$. In view of Theorem \ref{th:main}, it is sufficient to impose that zero is regular (it is neither an eigenvalue, nor a resonance).

On a related note, we mentioned that in \cite{DMSY-2017} the $L^p$-boundedness for the wave operator for the pair $(-\Delta_{\alpha,Y},-\Delta)$ was proved for all possible $p$'s, namely $p\in(1,3)$, under the implicit assumption of the absence of a zero eigenvalue as well as of positive resonances. The latter condition is now rigorously confirmed by Theorem \ref{th:main}.

This rises up two interesting open questions. First, one would like to investigate whether, in the spirit of \cite{DAncona-Pierfelice-Teta-2006,Iandoli-Scandone-2017}, an obstruction at zero in the form of a zero-energy eigenvalue or resonance would still allow one to derive certain (possibly slower) dispersive estimates. Let us recall that the counterpart problem for ordinary Schr\"odinger operators with spectral obstruction at zero has been intensively studied, significantly by Jensen and Kato \cite{Jensen-Kato-1979}, Rauch \cite{Rauch-1978}, Rodnianski and Schlag \cite{Rodnianski-Schlag-Invent2004}, Erdo\u{g}an, Burak, and Schlag \cite{Erdogan-Burak-Schlag-2004}, and Yajima \cite{Yajima-CMP2015_dispersive_est}.

Analogously, it would be of interest to understand whether some $L^p$-boundedness of the wave operator for $(-\Delta_{\alpha,Y},-\Delta)$ is still valid in the presence of a zero-energy eigenvalue, thus extending the recent results by Goldberg and Green \cite{Goldberg-Green-2016-thresholdsing} and by Yajima \cite{Yajima-DocMath2016,Yajma2016_3D_WaveOps_ThreshSing} on the $L^p$-boundedness of wave operators for the ordinary Schr\"odinger operators with threshold singularities.

These are questions that surely deserve to be investigated.


\def\cprime{$'$}

\end{document}